\documentclass[12pt,letterpaper]{amsart}
\usepackage{geometry}
\usepackage{amsmath}
\usepackage{amssymb}
\usepackage{amsthm}
\usepackage{setspace}

\newtheorem{theorem}{Theorem}[section]
\newtheorem{lemma}[theorem]{Lemma}
\newtheorem{corollary}[theorem]{Corollary}
\numberwithin{equation}{section}

\newcommand{\br}{\mathbf {R}}
\newcommand{\h}{\mathfrak{h}_{p,\,q}}
\newcommand{\g}{\mathfrak{g}}
\newcommand{\so}{\mathfrak{so}(p,\,q)}
\newcommand{\I}{\mathfrak{I}_{n}}
\newcommand{\hm}{{\rm{Hom}}}
\newcommand{\pari}{\tfrac{\partial}{\partial x^i}}
\newcommand{\parj}{\tfrac{\partial}{\partial x^j}}
\newcommand{\parn}{\tfrac{\partial}{\partial x^n}}
\newcommand{\parone}{\tfrac{\partial}{\partial x^1}}
\newcommand{\partwo}{\tfrac{\partial}{\partial x^2}}
\newcommand{\dpari}{dx^i}
\newcommand{\dparj}{dx^j}
\newcommand{\dparn}{dx^n}
\newcommand{\dparone}{dx^1}

\begin{document}
\author[]{SHITU FAWAZ JIMOH}
\title{Leibniz Cohomology with Adjoint Coefficients}
\maketitle

\textbf{ABSTRACT}: With the Poincaré group $\mathbf{R}^{3,\,1}\rtimes O(3,\,1)$ as the model of departure, we focus, for $n=p+q$, $p+q\ge 4$, on the affine indefinite orthogonal group $\mathbf{R}^{p,q}\rtimes SO(p,\,q)$. Denote by $\mathfrak{h}_{p,\,q}$ the Lie algebra of the affine indefinite orthogonal group. We compute the Leibniz cohomology of $\mathfrak{h}_{p,\,q}$ with adjoint coefficients, written $HL^*(\mathfrak{h}_{p,\,q};\,\mathfrak{h}_{p,\,q})$. We calculate several indefinite orthogonal invariants, and $\mathfrak{h}_{p,\,q} $-invariants and provide the Leibniz cohomology 
\section{INTRODUCTION}


In physics, Lie groups appear as symmetry groups of physical systems, and their Lie algebras (tangent vectors near the identity) may be thought of as infinitesimal symmetry motions. Thus Lie algebras and their representations are used extensively in physics, notably in quantum mechanics and particle physics. The indefinite orthogonal group $O(p,\,q)$ is among the most important groups that are broadly used in physics. Of particular interest in physics is the Lorentz group $O(3,1)$- the group of all Lorentz transformations of Minkowski spacetime, the classical and quantum setting for all (non-gravitational) physical phenomena \cite{hall2015lie} and the Poincaré group- the affine group of the Lorentz group $O(3,1)$, namely $\mathbf{R}^{3,1}\rtimes O(3,1)$. The Lorentz group is the setting for electromagnetism and special relativity. 


Classification of differentiable structures can be done with connections. This method involves calculation which sometimes can be unpleasant to work with. Connections on manifolds occur naturally as a cochain in the complex for Leibniz cohomology of vector fields with coefficients in the adjoint representation \cite{lodder2020leibniz}. In \cite{lodder1998leibniz} a definition of  Leibniz cohomology, $HL^*$, for differentiable manifolds was proposed. Consequently, we can reduce the problem of classifying differentiable structures using connections to computing Leibniz cohomology. This is one of many motivations for studying Leibniz algebras and Leibniz (co)homology. Leibniz (co)homolgy, $HL^*(\mathfrak{h},\,\mathfrak{h})$, is difficult to compute for an arbitrary Leibniz algebra $\mathfrak{h}$.

Leibniz cohomology with coefficients in the adjoint representation for semi-simple Leibniz algebras has been studied in 
\cite{Feldvoss-Wagemann}, while $HL^*( \mathfrak{h}; \, \mathfrak{h})$ has been used to study deformation theory 
in the category of Leibniz algebras \cite{Fil-Mandal}, \cite{Fil-Mandal-Muk}.

We offer a calculation of $HL(\h;\,\h)$, where $\h$ is the affine indefinite orthogonal Lie algebra, $p+q\ge4$. We compute this cohomology by identifying some indefinite orthogonal invariants in terms of balanced tensors. We then provide the Leibniz cohomology in terms of these invariants. The main tools used in this computation are the Hochschild-Serre spectral sequence and the Pirashvili spectral sequence. In low dimension $HL^0(\h;\,\h)=0$,  $HL^1(\h;\,\h) =\langle I \rangle$ and $HL^2(\h;\,\h) = \langle \rho \rangle$. Higher dimensions of $HL^*(\h;\,\h)$ contain echoes of these classes against a tensor algebra.  We prove in section \ref{Leibniz} that there is an isomorphism of graded vector spaces
$$HL^{*}(\h ;\, \h) \simeq \langle I,\rho \rangle \otimes T(\gamma_{pq}^{*}),$$
where $\langle I,\rho \rangle$ is the real vector space with basis $\{I,\,\rho\}$ and $T(\gamma_{pq}^{*})$ is the tensor algebra on the class of  $\gamma_{pq}^{*}$. We prove in section \ref{invariants} that $\gamma_{pq}^{*}$ and $\rho$ are $\so$-invariant. Also, $I$ is $\h$-invariant. 
 \begin{equation}
	\begin{split}
		I(\alpha_{ij}) &= 0,\ \ 1\le i<j \le p, \ \ p+1\le i<j \le n, \\
		I(\beta_{ij})& =0,\ \ 1\le i\le p,\  \ p+1\le j \le n,\\
		I(\pari)&=\pari, \ \ i=1,2,\cdots,n.
	\end{split}
\end{equation}
\begin{equation}
	\begin{split}
		\rho(\alpha_{ij} \, \wedge g )&=0, \ \ 1\le i<j \le p, \ \ p+1\le i<j \le n,\ \ {\rm{for\, all }}\, g \in \h ,\\
		\rho(\beta_{ij} \, \wedge g )&=0, \ \ 1\le i\le p ,\ \ p+1\le j \le n, \ \ {\rm{for\, all }}\, g \in \h , \\
		\rho(\pari \wedge \parj ) & = \alpha_{ij} , \ \ 1\le i\le j\le p,\\
		\rho(\pari \wedge \parj )& = - \alpha_{ij},\ \ p+1\le i\le j\le n,\\
		\rho(\pari \wedge \parj )&= \beta_{ij}, \ \ 1\le i\le p, \ \  p+1\le j\le n.		
	\end{split}
\end{equation}
\begin{equation}
	\begin{split}
		\gamma^{*}_{pq}=&\sum\limits_{ 1\le i<j \le p} (-1)^{i+j} \dparone \wedge  \cdots \wedge \widehat{\dpari} \wedge \cdots \wedge \widehat{\dparj} \wedge \cdots  \wedge \dparn \otimes \alpha^{*}_{ij}\\
		-&\sum\limits_{ p+1\le i<j \le n} (-1)^{i+j+1} \dparone \wedge  \cdots \wedge\widehat{\dpari}
		\wedge \cdots \wedge\widehat{\dparj} \wedge \cdots \wedge \dparn \otimes \alpha^{*}_{ij}\\
		-&\underset{p+1\le j \le n}{\sum\limits_{1\le i \le p}} (-1)^{i+j+1} \dparone \wedge  \cdots \wedge \widehat{\dpari} \wedge \cdots \wedge \widehat{\dparj} \wedge \cdots \wedge \dparn \otimes \beta^{*}_{ij}
	\end{split}
\end{equation}

This result generalizes Lodder's results \cite{lodder2020leibniz} obtained on the rotation groups. The rotation groups are a special case of  $\so$, the case when $q=0$, usually denoted by $\mathfrak{so}(p)$.

\section{THE INDEFINITE ORTHOGONAL LIE ALGEBRA}
Let
\begin{equation}
	I_{p,\,q} = 
	\begin{pmatrix}
		I_{p} & 0 \\
		0 &-I_{q} \\
	\end{pmatrix}
\end{equation} 

with $I_{k}$ denoting the $k \times k$ identity matrix. Then we define the indefinite orthogonal Lie algebra, $$\so =\{X \in M_{n}(\mathbf{R})\,| \,X^{T}I_{p,\,q}=-I_{p,\,q}X\}$$ with $ \,  n=p+q \,$ and $\, p, \,q \in$ \, \textbf{N}.

Consider the standard coordinates on $\mathbf{R^n}$ given by $(x_1,x_2,...,x_n)$  with unit vector fields $\frac{\partial}{\partial{x_i}}$ parallel to the $x_{i}$ axes.

Let
\begin{equation}
	\begin{split}
		\alpha_{ij} &:= x_{i} \parj - x_{j} \pari, \ \ \  1\le i<j \le p,  \ \ \ \, p+1\le i<j \le n,\\
		\beta_{ij} &:= x_{i} \parj + x_{j} \pari, \ \ \  1\le i \le p, \ \ \ p+1\le j \le n.
	\end{split}
\end{equation}
Then  $\{\alpha_{ij}\}, \, \{\beta_{ij}\},$ is a vector space basis for a lie algebra isomorphic to $\so$.  Let $\I$ be the $\mathbf{R}$ vector space spanned by $$\{\pari\}, \ \ \ i=1,2,\cdots,n.$$ Then $\I$ is an Abelian Lie algebra. Let $\h$ be the Lie algebra with basis given by the union of  $\{\alpha_{ij}\},\{\beta_{ij}\},\{\pari\}$ . Then $\I$ is an ideal of $\h$ and there is a short exact sequence of Lie algebras \cite[Page 217]{weibelintroduction}
$$0 \longrightarrow \I \longrightarrow \h  \longrightarrow \so \longrightarrow 0 $$ with $\h/\I$ $\simeq \so$ and $\h$ is an affine extension of $\so$.

Let $\alpha_{ij}^{*},\ \beta_{ij}^{*}$ be the dual of $\alpha_{ij}, \ \beta_{ij}$ respectively with respect to the basis \{$\alpha_{ij}$\} $\cup$
$ \{\beta_{ij}\}$ $\cup$ \{$\frac{\partial }{\partial x_{i}}$\} of $h_{p,\,q}$, and let $dx^{i}$ be the dual of $ \pari$.
\section{LIE ALGEBRA COHOMOLOGY WITH COEFFICIENTS}
Let $\g$ be a Lie algebra over a ring $k$ and $V$ be any $\g$-module. The Lie algebra cohomology of $\g$ with coefficients in the module $V$, written $H^*_{\rm{Lie}}(\g;\,V)$, is the cohomology of the cochain complex (Chevalley-Eilenberg complex), $$ \hm_{k}(\mathbf{R},\,V) \xrightarrow{\delta} \hm_{k}(\g,\,V) \xrightarrow{\delta} \hm_{k}(\g^{\wedge 2},\,V) \xrightarrow{\delta} \hm_{k}(\g^{\wedge 3},\,V) \xrightarrow{\delta}\cdots,$$ 
where $\g^{\wedge n}$ is the nth exterior power of $\g$ over $k$ and where the coboundary $\delta f$ of such an $n$-cochain is the $(n+1)$-cochain
\begin{equation}
	\begin{split}
		&\delta f(g_1\wedge ...  \wedge g_{n+1})=\sum_{i=1}^{n+1}(-1)^{i} g_{i}\cdot f(g_{1}\wedge...\hat{g_{i}}...\wedge g_{n+1}) \\ &+\sum_{ 1\le i<j\le n+1}^{}(-1)^{j}f(g_{1}\wedge...\wedge g_{i-1}\wedge[g_i,g_j]\wedge...\hat{g_{j}}... \wedge g_{n+1}).
	\end{split}
\end{equation}
In particular when $V=\g$, we obtain the Lie algebra cohomology with coefficients in the adjoint representation, written $H^*_{\rm{Lie}}(\g;\,\g)$ and when $V= g ':= \hm_{k}(\g,\,k)$, we obtain the Lie algebra cohomology with coefficients in the co-adjoint representation, written $H^*_{\rm{Lie}}(\g;\,\g ')$. We consider $k$ as a trivial $\g$-module, so $H^*_{\rm{Lie}}(\g;\, k)$ is the Lie algebra cohomology with trivial coefficients.\\ 
The (left) invariant submodule $V^{\g}$ of a $\g$-module $V$ is defined as:$$ V^{\g}=\{v\in V: \ xv=0 \ {\rm{ for \ all }} \ x \in \g \}$$
The action of $\g$ on $k$ is trivial as mentioned above. This means $g\cdot c= 0$ for all $c \in k$, for all $g \in \g$.
Let $\g$ act on itself via the adjoint representation. For $g$ $\in$ $\g$, the linear map 
ad$_{g}$: $\g$ $\xrightarrow{}$ $\g$ is given by $$g \cdot x :={\rm{ad}}_g(x)=[g,x],\ {\rm{for \ all \ }} x \in \g.$$

For $g \in \g$ and $f \in \hm_{k}(\g^{\wedge n}, \, V)$, we define the action of $\g$ on $\hm_{k}(\g^{\wedge n}, \, V)$ by

\begin{equation}
	\begin{split}
		&(gf)(x_1\wedge x_2 \wedge  ... \wedge x_n) = g\cdot f(x_1\wedge x_2 \wedge  ... \wedge x_n)\\
		&+ \sum^{n}_{i=1} f(x_1\wedge ...\wedge x_{i-1} \wedge[x_{i}, g] \wedge x_{i+1}\wedge ... \wedge g_n)
	\end{split}
\end{equation}

The action of $\g$ on itself extends to $\g^{\wedge n}$ by
\begin{equation}
	[g,\,x_1\wedge x_{2} \wedge ... \wedge x_n]= \sum_{i=1}^{n} x_1 \wedge _2 \wedge ... \wedge [g, \, x_i] \wedge ... \wedge x_{n}, \ {\rm{for}} \ g,\ x_{i} \in \g \ {\rm{for \ all \ i }}.
\end{equation}

\section{LEIBNIZ COHOMOLOGY WITH COEFFICIENTS}
Let $\g$ be a Leibniz algebra over a ring $k$ and $V$ be a representation $\g$. Let $$CL^n(\g,\, V) := \hm_{k}(\g^{\otimes n},\,V), \ \ {\rm{where}} \ \ n\ge 0.$$ The Leibniz cohomology of $\g$ with coefficients in the representation $V$, written $HL^*(\g;\,V)$, is the cohomology of the cochain complex $CL^*(\g,\, V)$ , $$ \hm_{k}(\mathbf{R},\,V) \xrightarrow{\delta} \hm_{k}(\g,\,V) \xrightarrow{\delta} \hm_{k}(\g^{\otimes 2},\,V) \xrightarrow{\delta} \hm_{k}(\g^{\otimes 3},\,V) \xrightarrow{\delta}\cdots,$$ 
where $\g^{\otimes n}$ is the nth tensor power of $\g$ over $k$ and where the coboundary $\delta f$ of such an $n$-cochain is the $(n+1)$-cochain
\begin{equation}
	\begin{split}
		&\delta f(g_1\otimes \cdots  \otimes g_{n+1}):= [g_1,\,  f(g_{2}\otimes \cdots \otimes g_{n+1}) ]\\ &+\sum_{i=2}^{n+1}(-1)^{i} [f(g_{1}\otimes \cdots \hat{g_{i}} \cdots \otimes g_{n+1}),\,g_{i}] \\ 
	   &+\sum_{ 1\le i<j\le n+1}^{}(-1)^{j+1}f(g_{1}\otimes \cdots \otimes g_{i-1}\otimes[g_i,g_j]\otimes \cdots \hat{g_{j}} \cdots  \otimes g_{n+1}).
	\end{split}
\end{equation}
It is easy to compute $HL^*(\g,\,V)$ in low dimension. For higher dimensions of $HL^{*}(\g,\,V)$, we use the Pirashvili spectral sequence \cite{pirashvili1994leibniz} and a long exact sequence  induced by the canonical projection map $\pi_{\rm{rel}} :\g^{\otimes(n+2)}  \xrightarrow{}\g^{\wedge(n+2)}, $ 	$\pi_{{\rm{rel}}}(g_{1}\otimes g_{2}\otimes...\otimes g_{n+2})=g_{1}\wedge g_{2}\wedge...\wedge g_{n+2}.$\\
\section{INVARIANTS FOR INDEFINITE ORTHOGONAL LIE ALGEBRA} \label{invariants}
We present the invariants of $\wedge ^{*}\I$, $\hm(\wedge^* \I, \h)$  and $\wedge^*\I \otimes \h$ under the action of $\so$ where $n=p+q.$
\begin{lemma}\label{invariantone}
There is a vector space isomorphism
\begin{equation}
[\wedge ^{*}\I]^{\so} \simeq  \mathbb{R} \oplus \langle v \rangle, where \ \ v = \parone  \wedge \partwo \wedge \cdots \wedge  \parn.
\end{equation}
\end{lemma}
\begin{proof}
Lemma \ref{invariantone} is proved in \cite{biyogmam2011leibniz} and can easily be found by direct calculations.
\begin{equation*}
	\begin{split}
	&[\I^{\wedge 0}]^{so(p,\,q)} = \mathbf{R},\quad [\I^{\wedge n}]^{so(p,q)} = \langle v\rangle \ {\rm{and}} \\
	&[\I^{\wedge k}]^{so(p,q)} = \{0\} \ \ {\rm{whenever}} \ \ k \notin \{0,n\}.
	\end{split}
\end{equation*}
\end{proof}

The following lemma is proved in \cite{biyogmam2013leibniz} and can also be found by direct calculations.

\begin{lemma} \label{invarianttwo} 
There is a vector space isomorphism
 
$[\I \otimes \h]^{\so} \simeq \langle I_{pq} \rangle$ where
\begin{equation*}
 \ \ I_{pq}=\sum\limits_{i=1}^p  \pari \otimes \pari -\sum\limits_{i=p+1}^n  \pari \otimes \pari,
\end{equation*}
$[\I^{\wedge2} \otimes \h]^{\so} \simeq \langle \rho_{pq} \rangle \ \ where$
\begin{equation*}
	\begin{split}
		&\rho_{pq} =\sum\limits_{ 1\le i<j \le p} \pari \wedge \parj \otimes \alpha_{ij} - \sum\limits_{ p+1\le i<j \le n}  \pari \wedge \parj \otimes \alpha_{ij}\\
		&-\underset{p+1\le j \le n}{\sum\limits_{ 1\le i\le p}} \pari \wedge \parj \otimes \beta_{ij},	
	\end{split}
\end{equation*}
$[\I^{\wedge(n-1)} \otimes \h]^{\so} \simeq \langle \beta_{pq} \rangle \ \  where$
\begin{equation*}
	\begin{split}
		\beta_{pq}=&\sum\limits_{i=1}^p (-1)^{i+1} \parone  \wedge \cdots  \wedge \hat{\pari}  \wedge \cdots \wedge \parn \otimes \pari \\
		- &\sum\limits_{i=p+1}^n (-1)^{i} \parone  \wedge \cdots \wedge \hat{\pari}  \wedge \cdots \wedge \parn \otimes \pari,
	\end{split}
\end{equation*}
$[\I^{\wedge(n-2)} \otimes \h]^{\so} \simeq \langle \gamma_{pq} \rangle\ \  where$
\begin{equation*}
	\begin{split}		
		\gamma_{pq}=&\sum\limits_{ 1\le i<j \le p} (-1)^{i+j} \parone \wedge  ... \wedge \hat{\pari} \wedge ... \wedge \hat{\parj} \wedge...  \wedge \parn \otimes \alpha_{ij}\\
		-&\sum\limits_{ p+1\le i<j \le n} (-1)^{i+j+1} \parone \wedge  ... \wedge\hat{\pari} \wedge ... \wedge\hat{\parj} \wedge...  \wedge \parn \otimes \alpha_{ij}\\
		-&\underset{p+1\le j \le n}{\sum\limits_{1\le i \le p}} (-1)^{i+j+1} \parone \wedge  ... \wedge \hat{\pari} \wedge ... \wedge \hat{\parj} \wedge...  \wedge \parn \otimes \beta_{ij}.
	\end{split}
\end{equation*}

$[\I^{\wedge k} \otimes \h]^{\so} \simeq \{0\}$ for $k\notin  \{1,2,n-1, n-2\}$, 
\end{lemma}
\begin{lemma} \label{invariant3}
	There is a vector space isomorphism
\begin{equation}
	\hm(\I^{\wedge k},\, \h) \simeq \I^{\wedge k} \otimes \h\ ,\ for  \ \ k=0,1,2, \cdots.
\end{equation}
\end{lemma}
\begin{proof}
We define $\psi:\hm(\I^{\wedge k},\, \h)\longrightarrow \I^{\wedge k} \otimes \h$,
\begin{equation}
	\begin{split}
		\psi(\phi) :=& \sum_{s=0}^{k} \underset{p+1\le i_{k-s+1} < i_{k-s+2}< \cdots < i_{k} \le n}{\sum\limits_{1\le i_{1} < i_{2} < \cdots < i_{k-s} \le p}} (-1)^{s} z \otimes \phi(z),
	\end{split}
\end{equation}
 for all $z=  \tfrac{\partial}{\partial x^{i_1}}
 \wedge \tfrac{\partial}{\partial x^{i_2}} \wedge
 \ldots
 \wedge \frac{\partial}{\partial x^{i_k}} \in\I^{\wedge k}$ and $\phi \in \hm(\I^{\wedge k}, \h)$.\\
$\psi$ is an isomorphism and $\so$-equivariant. The proof of isomorphism is straightforward.  To show the equivariant property, we start by looking at the action of $\so$ on $\hm(\I^{\wedge k },\, \h)$ and $\I^{\wedge k } \otimes \h$.
For $g \in \so$ and $\phi \in \hm(\I^{\wedge k},\, \h) $, the action of $g$ on $\phi$ is given by;
\begin{equation}
	(g\cdot \phi)(z)= [g,\,\phi(z)] + \phi([z,\,g])
	=  [g,\,\phi(z)] -  \phi([g,\,z]), {\rm{\ for \ all }} \ z \in \I^{\wedge k}. 
\end{equation}
By defintion, we have 
\begin{equation*}
	\begin{split}
		\psi(g\cdot\phi)&= \sum_{s=0}^{k} \underset{p+1\le i_{k-s+1} < i_{k-s+2}< \cdots < i_{k} \le n}{\sum\limits_{1\le i_{1} < i_{2} < \cdots < i_{k-s} \le p}} (-1)^{s} z \otimes (g\cdot\phi)(z),\\
		&= \sum_{s=0}^{k} \underset{p+1\le i_{k-s+1} < i_{k-s+2}< \cdots < i_{k} \le n}{\sum\limits_{1\le i_{1} < i_{2} < \cdots < i_{k-s} \le p}} (-1)^{s} z \otimes [g,\,\phi(z)]\\
		&- \sum_{s=0}^{k} \underset{p+1\le i_{k-s+1} < i_{k-s+2}< \cdots < i_{k} \le n}{\sum\limits_{1\le i_{1} < i_{2} < \cdots < i_{k-s} \le p}} (-1)^{s} z \otimes \phi([g,\,z])
	\end{split}  
\end{equation*}
\begin{equation*}
	\begin{split}
		g \cdot \psi(\phi) = \sum_{s=0}^{k} \underset{p+1\le i_{k-s+1} < i_{k-s+2}< \cdots < i_{k} \le n}{\sum\limits_{1\le i_{1} < i_{2} < \cdots < i_{k-s} \le p}} (-1)^{s} [g,\,z] \otimes \phi(z)\\
		+ \sum_{s=0}^{k} \underset{p+1\le i_{k-s+1} < i_{k-s+2}< \cdots < i_{k} \le n}{\sum\limits_{1\le i_{1} < i_{2} < \cdots < i_{k-s} \le p}} (-1)^{s} z \otimes [g,\,\phi(z)]
	\end{split}	
\end{equation*}

We show for all $g \in \so$ and $z \in\I^{\wedge k} $
\begin{equation*}
	\begin{split}
		&\sum_{s=0}^{k} \underset{p+1\le i_{k-s+1} < i_{k-s+2}< \cdots < i_{k} \le n}{\sum\limits_{1\le i_{1} < i_{2} < \cdots < i_{k-s} \le p}} (-1)^{s} [g,\,z] \otimes \phi(z)\\
		=&- \sum_{s=0}^{k} \underset{p+1\le i_{k-s+1} < i_{k-s+2}< \cdots < i_{k} \le n}{\sum\limits_{1\le i_{1} < i_{2} < \cdots < i_{k-s} \le p}} (-1)^{s} z \otimes \phi([g,\,z])
	\end{split}
\end{equation*}
and then conclude $\psi(g\cdot\phi) = g \cdot \psi(\phi)$.
\end{proof}
The details of Lemma \ref{invariant3} are in my dissertation. 

\begin{corollary}\label{corollary}
There is a vector space isomorphism 
\begin{equation}
	[\hm(\I^{\wedge k},\, \h)]^{\so} \simeq [\I^{\wedge k} \otimes \h]^{\so}, \ for \ all \ k. 
\end{equation}
\end{corollary}
\begin{proof}
The following results follow because $\psi$ in Lemma \ref{invariant3} is $\so$-equivariant. The invariants [$\hm(\I^{\wedge*},\, \h)]^{\so}$ can be found by direct calculations although this is difficult. 

For $k=1$ , $ [\hm(\I,\, \h)]^{\so} = \langle I \rangle$, 
\begin{equation}
	\psi(I) :=\sum\limits_{i=1}^p  \pari  \otimes I(\pari)
	-\sum\limits_{i=p+1}^n \pari  \otimes I(\pari) = I_{pq}.
\end{equation}

where $$I(\pari)=\pari \, , i=1,2,\cdots,n. $$
Note that $  \langle I \rangle =  \langle -I \rangle $. 

For $k = 2$,  $ [\hm(\I^{\wedge2},\, \h)]^{\so} = \langle \rho \rangle$,
\begin{equation}
	\begin{split}
		\psi(\rho):=&\sum\limits_{ 1\le i<j \le p}  \pari \wedge \parj  \otimes \rho( \pari \wedge \parj ) \,\\
		+&\sum\limits_{ p+1\le i<j \le n}\pari \wedge \parj   \otimes \rho( \pari \wedge \parj )\\
		-&\underset{p+1\le j \le n}{\sum\limits_{ 1\le i\le p}} \pari \wedge \parj   \otimes \rho( \pari \wedge \parj ) = \rho_{pq} \\
		\end{split}
\end{equation}
where \begin{equation}
	\begin{split}
		\rho(\pari \wedge \parj ) & = \alpha_{ij} , \, 1\le i\le j\le p,\\
		\rho(\pari \wedge \parj )& = - \alpha_{ij},\, p+1\le i\le j\le n,\\
		\rho(\pari \wedge \parj )&= \beta_{ij},\, 1\le i\le p, \, p+1\le j\le n.		
	\end{split}
\end{equation}

For $k = n-1$, $[\hm(\I^{\wedge(n-1)},\, \h)]^{\so} = \langle \beta \rangle$,
\begin{equation}
	\begin{split}
		\psi(\beta) :&= \sum\limits_{i=1}^p \parone  \wedge \cdots  \wedge \hat{\pari}  \wedge \cdots \wedge \parn  \otimes \beta(\parone  \wedge \cdots  \wedge \hat{\pari}  \wedge \cdots  \wedge \parn)\\
		&- \sum\limits_{i=p+1}^n \parone  \wedge \cdots  \wedge \hat{\pari}  \wedge \cdots  \wedge \parn \otimes \beta(\parone  \wedge \cdots   \wedge \hat{\pari}  \wedge \cdots  \wedge \parn)\\		
		& = \beta_{pq}  
	\end{split}
\end{equation}
where $$\beta(\parone \wedge \cdots \wedge \hat{\pari} \wedge \cdots  \wedge \parn)= (-1)^{i+1}\pari,i= 1,2, \cdots , p $$
$$\beta(\parone \wedge \cdots  \wedge \hat{\pari} \wedge \cdots  \wedge \parn)= (-1)^{i}  \pari ,i= p+1, \cdots  , n.$$ 
For $k = n-2$, $[\hm(\I^{\wedge(n-2)},\, \h)]^{\so} = \langle \gamma \rangle$,
\begin{equation}
	\begin{split}
		\psi(\gamma) &:= \\
		&\sum\limits_{ 1\le i<j \le p} \parone \wedge  \cdots \hat{\pari} \cdots \hat{\parj} \cdots  \wedge \parn \otimes \gamma(\parone \wedge  \cdots \hat{\pari} \cdots \hat{\parj} \cdots  \wedge \parn)\\
		-&\sum\limits_{ p+1\le i<j \le n} \parone \wedge \cdots  \hat{\pari} \cdots \hat{\parj} \cdots   \wedge \parn \otimes \gamma(\parone \wedge  \cdots \hat{\pari} \cdots \hat{\parj} \cdots   \wedge \parn)\\
		-&\underset{p+1\le j \le n}{\sum\limits_{1\le i \le p}} \parone \wedge \cdots  \hat{\pari} \cdots \hat{\parj} \cdots  \wedge \parn \otimes \gamma(\parone \wedge \cdots \hat{\pari} \cdots \hat{\parj} \cdots  \wedge \parn)\\
		& = \gamma_{pq}
	\end{split}
\end{equation}
where
\begin{equation}
	\begin{split}
		\gamma(\parone \wedge \cdots \hat{\pari} \cdots \hat{\parj} \cdots \wedge \parn) & = (-1)^{i+j} \alpha_{ij},\ 1\le i \le j \le p+1,\\
		\gamma(\parone \wedge \cdots \hat{\pari} \cdots \hat{\parj} \cdots \wedge \parn)& = (-1)^{i+j+1} \alpha_{ij}, \ p+1\le i \le j \le n,\\
		\gamma(\parone \wedge \cdots \hat{\pari} \cdots \hat{\parj} \cdots \wedge \parn) & = (-1)^{i+j+1} \beta_{ij}, \ 1 \le i \le p, \, p+1 \le j\le n.
	\end{split}
\end{equation}
$[\hm(\I^{\wedge k},\, \h)]^{\so} = \{0\}, $ whenever $ k\notin \{1,2,n-1, n-2\}$.
\end{proof}

\section{LIE ALGEBRA COHOMOLOGY OF $\h$ WITH COEFFICIENTS}
Let $\g$ be a Lie algebra over a ring $k$ and $V$ be any $\g$-module.
For each $n\geq0$, let  $\pi_{R}:\g\otimes \g^{\wedge(n+1)} \xrightarrow{} \g^{\wedge(n+2)}$  be the canonical projection $\pi_{R}(g_{1}\otimes g_{2}\wedge \cdots \wedge g_{n+2})=g_{1}\wedge g_{2}\wedge...\wedge g_{n+2}$, and $\pi^*_R:{\rm{Hom}}(\g^{\wedge(n+2)},V)\xrightarrow{} {\rm{Hom}}(\g \otimes \g^{\wedge(n+1)}, V)$ be the map induced by $\pi_{R}$.

Let $CR^n(\g)={\rm{coker}}[{\rm{Hom}}(\g^{\wedge(n+2)},V)\xrightarrow{\pi^*_{R}} {\rm{Hom}}(\g \otimes \g^{\wedge(n+1)}, V)]$ and  $HR^*(\g)$ be the cohomology of the complex  $CR^n(\g)$.

There is a short exact sequence
\begin{equation}
	\begin{split}
	0&\xrightarrow{} {\rm{Hom}}(\g^{\wedge(n+2)},V)\xrightarrow{\pi^*_R} {\rm{Hom}}(\g\otimes \g^{\wedge(n+1)}, V)\xrightarrow{} CR^n(\g) \xrightarrow{}0
	\end{split}
\end{equation}
which induces a long exact sequence
\begin{equation}
	\begin{split} \label{les1}
		\cdots \xrightarrow{}H^{n+2}_{\rm{Lie}}(\g;V)&\xrightarrow{\pi^*_{R}} H^{n+1}_{\rm{Lie}}(\g;\g^{'})\xrightarrow{} HR^{n}(\g) \xrightarrow{c_R}\\
		H^{n+3}_{\rm{Lie}}(\g;V)&\xrightarrow{\pi^*_{R}} H^{n+2}_{\rm{Lie}}(\g;\g^{'})\xrightarrow{} \cdots 
	\end{split}
\end{equation}
where $c_R$ is the connecting homomorphism and $\g^{'}= {\rm{Hom}}(\g, \,V ) $.

In this section my calculation is done with  $k=V=\br$.

Let $\alpha_{ij}^{*}$, $\beta_{ij}^{*}$ be the dual of $\alpha_{ij}$,  $\beta_{ij}$ respectively with respect to the basis \{$\alpha_{ij}$\} $\cup$
$ \{\beta_{ij}\}$ $\cup$ $\{\pari\}$ of $\h$, and let $dx^{i}$ be the dual of $\pari$. Let
\begin{equation}
	\begin{split}
		\gamma_{pq}^*=&\sum\limits_{ 1\le i<j \le p} (-1)^{i+j} dx^1 \wedge \cdots  \wedge \widehat{dx^i} \wedge \cdots  \wedge \widehat{dx^j} \wedge \cdots  \wedge dx^n \otimes \alpha_{ij}\\
		-&\sum\limits_{ p+1\le i<j \le n} (-1)^{i+j+1}dx^1 \wedge \cdots  \wedge \widehat{dx^i} \wedge \cdots  \wedge \widehat{dx^j} \wedge \cdots  \wedge dx^n \otimes \alpha_{ij}\\
		-&\underset{p+1\le j \le n}{\sum\limits_{1\le i \le p}} (-1)^{i+j+1} dx^1 \wedge \cdots  \wedge \widehat{dx^i} \wedge \cdots  \wedge \widehat{dx^j} \wedge \cdots  \wedge dx^n \otimes \beta_{ij}
	\end{split}
\end{equation}
\begin{lemma}\label{Serre} 
There is an isomorphism in Lie-algebra cohomology
\begin{equation}
	\begin{split}
		H^{*}_{\rm{Lie}}(\h;\mathbf{R})\simeq H^{*}_{\rm{Lie}}(\so);\mathbf{R}) \otimes [H^{*}_{\rm{Lie}}(\I;\mathbf{R})]^{\so}\\
		H^{*}_{\rm{Lie}}(\h;\,\h^{'})\simeq H^{*}_{Lie}(\so;\,\mathbf{R}) \otimes [H^{*}_{\rm{Lie}}(\I;\,\h^{'})]^{\so}
	\end{split}
\end{equation} 
where $\h^{'}={\rm{Hom}}(\h, \,\mathbf{R} ).$
\end{lemma}
\begin{proof}
	 We apply the Hochschild-Serre spectral sequence \cite{hochschild1953cohomology} to
	the ideal $\I$ of $\h$ and use the isomorphism of Lie algebras $\h/\I \simeq \so$.
\end{proof}
From section section \ref{invariants}, $[\I^{\wedge 0}]^{\so} = \mathbf{R},$ 
$[\I^{\wedge n}]^{\so} = \langle v \rangle \,,{\rm{where}} \, v = \parone \wedge \partwo \wedge ... \wedge \parn,$ and
$[\I^{\wedge k}]^{\so} = \{0\}, \,k \notin \{0,n\}.$

\begin{theorem}\label{Lieideal}
	 For $p+q\ge 4$,
	 \begin{equation}
	 	H^{*}_{\rm{Lie}}(\h;\, \mathbf{R})\simeq H^{*}_{\rm{Lie}}(\so;\,\br) \otimes (\br \oplus \langle v^* \rangle ),
	 \end{equation}
where $v^* = dx^{1} \wedge dx^2 \wedge \cdots \wedge dx^n$.
 \end{theorem}

\begin{proof}
$[H^{*}_{\rm{Lie}}(\I;\,\mathbf{R})]^{\so}$ is the cohomology of the cochain complex $[\hm(\wedge ^{*}\I,\,\mathbf{R} )]^{\so}$. Note that $$[{\rm{Hom}}(\wedge ^{*}\I,\,\mathbf{R} )]^{\so} \simeq {\rm{Hom}}
([\wedge ^{*}\I]^{\so},\,\mathbf{R}).$$ So we will compute $[H^{*}_{\rm{Lie}}(\I;\,\mathbf{R})]^{\so}$ from the complex:
\begin{equation}
	\begin{split}
		&\hm([\I^{\wedge0}]^{\so},\,\mathbf{R})\xrightarrow{\delta}\hm([\I]^{\so},\,\mathbf{R})\xrightarrow{\delta}\hm([\I^{\wedge2}]^{\so},\,\mathbf{R})\xrightarrow{\delta} \cdots\\
		&\xrightarrow{\delta}\hm([\I^{\wedge n}]^{\so},\,\mathbf{R})\xrightarrow{\delta} 0. 
	\end{split}
\end{equation} 

[$H^{0}_{\rm{Lie}}(\I;\,\br)]^{\so} \simeq \mathbf{R}$, [$H^{n}_{\rm{Lie}}(\I;\,\br)]^{\so} \simeq {\hm}( \langle v \rangle \,\,\br)$ and

[$H^{k}_{\rm{Lie}}(\I;\,\mathbf{R}) ]^{\so}\simeq \{0\}$ for $k \notin \{0,n\}$. As a consequence, $[H^{*}_{\rm{Lie}}(\I;\,\br)]^{\so} \simeq \br \oplus \langle v^* \rangle $.
By Lemma \ref{Serre}, $H^{*}_{\rm{Lie}}(\h;\mathbf{R})\simeq H^{*}_{\rm{Lie}}(\so;\br) \otimes (\br \oplus \langle v^* \rangle )$.
\end{proof}
\begin{theorem} 
	There is an isomorphism in Lie-algebra cohomology
\begin{equation}
	H^{*}_{\rm{Lie}}(\h;\,\h^{'}) \simeq H^{*}_{\rm{Lie}}(\so;\,\mathbf{R}) \otimes \langle\beta^{*}_{pq},\ \gamma^{*}_{pq}\rangle
\end{equation}
\end{theorem}
\begin{proof}
$[H^{*}_{\rm{Lie}}(\I;\, \h^{'}]^{\so}$ is the cohomology of the cochain complex $[\hm(\wedge ^{*}\I,\,\h^{'})]^{\so}$.
Note that 
\begin{equation*}
	\begin{split}
	[\hm(\wedge ^{*}\I,\,\h^{'})]^{\so} &= [\hm(\wedge ^{*}\I,\,\hm (\h, \,\mathbf{R})]^{\so}\\
	&\simeq [\hm(\wedge ^{*}\I \otimes \h ,\,\mathbf{R} )]^{\so}\\
	&\simeq \hm([\wedge ^{*}\I \otimes \h ]^{\so},\,\mathbf{R})
	\end{split}
\end{equation*} 

We use some invariants provided in section \ref{invariants} to compute the Lie-algebra homology groups $[H^{\rm{Lie}}_{*}(\I;\h)]^{\so}$ and then apply the universal coefficient theorem to find $[H^{*}_{\rm{Lie}}(\I;\, \h^{'})]^{\so}$.

Lie-algebra homology groups $[H^{\rm{Lie}}_{*}(\I;\h)]^{\so}$ can be computed from the complex:
\begin{equation}
	\begin{split}
		[\I^{\wedge0} \otimes \h]^{\so}&\xleftarrow{\delta}[\I \otimes \h]^{\so}\xleftarrow{\delta}[\I^{\wedge2}\otimes \h]^{\so} \xleftarrow{\delta} \cdots\\
		\cdots&\xleftarrow{}[\I^{\wedge n-2 }\otimes \h]^{\so} \xleftarrow{\delta} [\I^{\wedge n-1 }\otimes \h]^{\so} \xleftarrow{} 0.
	\end{split}
\end{equation}

Recall that $[\I^{\wedge0} \otimes \h]^{\so}\simeq \{0\}$, $[\I \otimes \h]^{\so} \simeq \langle I_{pq} \rangle$, $[\I^{\wedge2} \otimes \h]^{\so} \simeq \langle \rho_{pq} \rangle $, $[\I^{\wedge(n-1)} \otimes \h]^{\so} \simeq \langle \beta_{pq} \rangle$ and $[\I^{\wedge(n-2)} \otimes \h]^{\so} \simeq \langle \gamma_{pq} \rangle$. \\

Since $\delta(I_{pq})=0$, $\delta(\rho_{pq}) = I_{pq} $, $\delta(\beta_{pq})=0, \, \delta(\gamma_{pq})= 0$ it follows that  $$[H^{\rm{Lie}}_0 (\I;\, \h)]^{\so}= \{0\},$$ $$[H^{\rm{Lie}}_1 (\I;\, \h)]^{\so} =\langle I_{pq} \rangle / \langle I_{pq} \rangle \simeq \{0\},$$ $$[H^{\rm{Lie}}_2 (\I;\, \h)]^{\so}= \{0\}, $$ $$[H_{n-1}^{\rm{Lie}}(\I; \,\h)]^{\so} = \langle \beta_{pq} \rangle, $$ and  $$[H_{n-2}^{\rm{Lie}}(\I; \,\h)]^{\so} = \langle \gamma_{pq} \rangle.$$
Consequently,  $[H^{\rm{Lie}}_{*}(\I;\h)]^{\so}$  has two classes; $\langle \beta_{pq} \rangle$ and $\langle \gamma_{pq}  \rangle $, and  $$[H^{\rm{Lie}}_{*}(\I;\h)]^{\so} \simeq  \langle \beta_{pq}, \gamma_{pq}  \rangle.$$
By the universal coefficient theorem, $$[H^{*}_{\rm{Lie}}(\I;\, \h^{'})]^{\so} \simeq \hm(\ [H^{\rm{Lie}}_{*}(\I;\h)]^{\so}
;\, \mathbf{R}).$$
Therefore, $$[H^{*}_{\rm{Lie}}(\I;\, \h^{'})]^{\so} \simeq  \langle \beta^{*}_{pq}, \gamma^{*}_{pq} \rangle, $$where $ \beta^{*}_{pq}$ and $\gamma^{*}_{pq} $ are the dual of  $ \langle \beta_{pq} \rangle, \langle \gamma_{pq} \rangle $ respectively.\
By Lemma \eqref{Serre} , $$H^{*}_{\rm{Lie}}(\h;\,\h^{'}) \simeq H^{*}_{\rm{Lie}}(\so;\,\mathbf{R}) \otimes \langle\beta^{*}_{pq},\ \gamma^{*}_{pq}\rangle. $$
\end{proof}

\subsection{THE $HR^*(\h)$ COHOMOLOGY GROUPS}
\begin{theorem} 
	 For $m\ge 0$, there is a vector space isomorphism
\begin{equation}
	HR^m(\h) \simeq H^{m+3}_{\rm{Lie}}(\so; \,\mathbf{R}) \oplus (H^{m+3-n}_{\rm{Lie}}(\so; \, \mathbf{R}) \otimes \langle \gamma^*_{pq} \rangle).
\end{equation}	
\end{theorem}
\begin{proof}
	We compute $HR^*(\h)$ from the long exact sequence \eqref{les1} 
\begin{equation*}
		\begin{split}
			HR^{m}(\h) &\xrightarrow{c_R}H^{m+3}_{\rm{Lie}}(\h;\br)\xrightarrow{\pi^*_{R}} H^{m+2}_{\rm{Lie}}(\h;\h^{'})\xrightarrow{}\\
			HR^{m+1}(\h)&\xrightarrow{c_R} H^{m+4}_{\rm{Lie}}(\h;\br)\xrightarrow{} H^{m+3}_{\rm{Lie}}(\h;\h^{'}) \xrightarrow{} \cdots 
		\end{split}
\end{equation*}
For $ \theta \in H^{m+3}_{\rm{Lie}}(\so;\, \mathbf{R})$, $\pi^*_R(\theta)= 0 \  \rm{in}   \  H^{m+2}_{\rm{Lie}}(\h;\,\h^{'})$. Consequently, we have  $\pi^*_{R}[H^{m+3}_{\rm{Lie}}(\so;\mathbf{R})]= 0$, since $\theta$ is arbitrary.\\
For $\theta \, \otimes \, v^* \in H^{m+3-n}_{\rm{Lie}}(\so;\mathbf{R}) \otimes \langle v^* \rangle$, $\pi^*_R(\theta \, \otimes \, v^*)=\theta \otimes \beta^*_{pq} \  \rm{ in } \ H^{m+2}_{\rm{Lie}}(\h;\,\h^{'}),$ so  
$\theta \, \otimes \,\beta^*_{pq}$ will be trivial in $HR^{m+1}(\h)$ and $\theta \, \otimes \, \gamma^*_{pq} \in {\rm{Im}}[H^{m+2}_{\rm{Lie}}(\h;\h^{'})\xrightarrow{} HR^{m+1}(\h)]$. 

We conclude from the above that ${\rm{Im}}c_R={\rm{ker}}\pi^*_{R} \simeq H^{m+3}_{\rm{Lie}}(\so; \,R)$ and ${\rm{coker}}\pi^*_{R} \simeq H^{m+3-n}_{\rm{Lie}}(\so; \, \boldsymbol{R}) \otimes \langle \gamma^*_{pq} \rangle$.

There are also two classes in $HR^m(\h)$;
\begin{enumerate}
	\item $\theta'$, where $c_R(\theta')=\theta \in  H^{m+3}_{\rm{Lie}}(\so, \boldsymbol {R})$ (Those in ${\rm{ker}}\pi^*_{R}$)
	\item $\theta \, \otimes \, \gamma^*_{pq}$, where $\theta \in H^{i}_{\rm{Lie}}(\so, \boldsymbol {R})$ and $i=0,1,2,\cdots $ (Those in coker$\pi^*_{R}$).
\end{enumerate}
Finally, 
\begin{equation*}
	HR^m(\h) \simeq H^{m+3}_{\rm{Lie}}(\so; \,R) \oplus (H^{m+3-n}_{\rm{Lie}}(\so; \, \mathbf{R}) \otimes \langle \gamma^*_{pq} \rangle), {\rm{where}}\, m\geq 0.
\end{equation*}	
\end{proof}

\section{LEIBNIZ COHOMOLOGY OF $\h$ WITH COEFFICIENTS} \label{Leibniz}
Let $\g$ be a Lie algebra over a ring $k$ and $V$ be any $\g$-module. For each $n\geq 0$,  let $\pi_{\rm{rel}} :\g^{\otimes(n+2)}  \xrightarrow{}\g^{\wedge(n+2)}$ be the canonical projection map 	$\pi_{{\rm{rel}}}(g_{1}\otimes g_{2}\otimes \ldots \otimes g_{n+2})=g_{1}\wedge g_{2}\wedge \ldots \wedge g_{n+2}, $ and $\pi^*_{{\rm{rel}}}:{\rm{Hom}}(\g^{\wedge(n+2)},V) \longrightarrow {\rm{Hom}}(\g ^{\otimes(n+2)},V )$ be the map induced by $\pi_{{\rm{rel}}}$.\\
We define $C_{{\rm{rel}}}^n(\g) := {\rm{Coker}}[\hm (\g ^{\wedge(n+2)},\,V )  \xrightarrow{\mathit{\pi^*_{\rm{rel}}}} \hm (\g ^{\otimes(n+2)},\,V)]$ and $H^*_{\rm{rel}}(\g)$ be the cohomology of the complex $C^*_{{\rm{rel}}}(\g)$.
	
There is a short exact sequence  
$$0 \longrightarrow {\rm{Hom}}(\g ^{\wedge(n+2)},\, V )  \xrightarrow{\mathit{\pi^*_{\rm{rel}}}} {\rm{Hom}}(\g ^{\otimes(n+2)},\,V) \longrightarrow C_{\rm{rel}}^n(\g ) \longrightarrow 0$$
which induces a long exact sequence of cohomology 
\begin{equation}\label{les2}
	\begin{split}
		\cdots \longrightarrow H^{n+2}_{\rm{Lie}}(\g ;\,V)  &\xrightarrow{\mathit{\pi^*_{\rm{rel}}}} HL^{n+2}(\g ;\,V) \xrightarrow{} H_{rel}^{n}(\g) \xrightarrow{\mathit{c_{\rm{rel}}}}\\
		H^{n+3}_{Lie}(\g ;\,V )  &\xrightarrow{\mathit{\pi^*_{\rm{rel}}}}  HL^{n+3}(\g ;\,V ) \longrightarrow H_{\rm{rel}}^{n+1}(\g ) \xrightarrow{\mathit{c_{\rm{rel}}}} \cdots
	\end{split}
\end{equation}

where $c_{\rm{rel}}$ is the connecting homomorphism.


\subsection{PIRASHVILI SPECTRAL SEQUENCE} 
\begin{theorem}
: Let $\g$ be a Lie algebra over a field $\bf{F}$ and let $V$ be a left $\g$-module.  Then there is a first-quadrant spectral sequence converging to $H^*_{\rm{rel}}(\g ; \, V)$ with
$$  E_2^{m, \, k} \simeq HR^m (\g) \otimes HL^k ( \g ; \, V), \ \ \ m \geq 0, \ \ \ k \geq 0 ,  $$
provided that $HR^m ( \g)$ and $HL^k ( \g ; \, V)$ are finite dimensional vector spaces in each dimension.  
If this finite condition is not satisfied, then the completed tensor product $\widehat {\otimes}$ can be used.
\end{theorem}
We outline the key features of the construction and introduce notation that will be used in the sequel.  Let
$A^{m, \, k}$ denote those elements $f \in {\rm{Hom}}( \g^{\otimes (k + m + 2)}, \; V)$ that are skew-symmetric in the last
$m+1$ tensor factors of $\g^{\otimes (k + m + 2)}$.  Filter the complex $C^*_{\rm{rel}}$ by 
$$  F^{m, \, k} =   A^{m, \, k} / {\rm{Hom}}( \Lambda^{k+m+2} ( \g) , \, V) .  $$
Then $F^{m, \, *}$ is a subcomplex of $C^*_{\rm{rel}}$ with $F^{0, \, *} = C^*_{\rm{rel}}$ and  
$F^{m+1, \, *}  \subseteq F^{m, \, *}$.  To identify the $E^{*, \, *}_0$ term, use the isomorphism

\begin{equation}  \label{switch}
	{\rm{Hom}}( \g^{\otimes (k+m+2)}, \, V) \simeq  
	{\rm{Hom}}( \g^{\otimes (m+2)}, \, {\rm{Hom}} ( \g^{ \otimes k}, \, V)) 
\end{equation}
Then 
\begin{equation}
	\begin{split}
		E^{m, \, k}_0 & = F^{m, \, k}/ F^{m+1, \, k-1}  \\
		& \simeq {\rm{Hom}} ( \g \otimes \g^{\wedge (m+1)} / \g^{\wedge (m+2)} , \, 
		{\rm{Hom}}( \g^{\otimes k}, \, V) ), 
	\end{split}
\end{equation}
and $d_0^{m, \, k} : E^{m, \, k}_0 \to E^{m, \, k+1}_0$, $m \geq 0$, $k \geq 0$.  It follows that
$$  E^{m, \, k}_1 \simeq {\rm{Hom}}( \g \otimes \g^{\wedge (m+1)} / \g^{\wedge (m+2)} , \, 
HL^k ( \g; \, V)) .  $$
Now, $d_1^{m, \, k} : E^{m, \, k}_1 \to E^{m+1, \, k}_{1}$.  Since the action of $\g$ on $HL^* ( \g; \, V)$
is trival, we have $E^{m, \, k}_2 \simeq HR^m ( \g ) \widehat{\otimes} HL^k ( \g; \, V)$.  Using
the isomorphism \eqref{switch}, we consider an element of $E^{m, \, k}_2$ operationally in the form
$HL^k ( \g; \, V) \widehat{\otimes} HR^m ( \g )$.

In this section, all computations are done with $V=\g = \h$ and $k=\br$.

\begin{lemma}\label{Serre2} There is an isomorphism in Lie-algebra cohomology 
\begin{equation}
		H^{*}_{\rm{Lie}}(\h;\,\h)\simeq H^{*}_{\rm{Lie}}(\so;\,\br) \otimes [H^{*}_{\rm{Lie}}(\I;\,\h)]^{\so}
	\end{equation}
\end{lemma}
\begin{proof}
We apply the Hochschild-Serre spectral sequence \cite{hochschild1953cohomology} to the ideal $\I$ of $\h$ and use the isomorphism of Lie algebras $\h/\I \simeq \so$.
\end{proof}


\begin{theorem} There is an isomorphism in Lie-algebra cohomology 
\begin{equation}
	 H^{*}_{\rm{Lie}}(\h;\,\h)\simeq H^{*}_{\rm{Lie}}(\so;\,\br) \otimes \langle I, \rho \rangle.
\end{equation}
\end{theorem}
\begin{proof}
	$[H^{*}_{\rm{Lie}}(\I;\,\h)]^{\so}$ is the cohomology of the complex $ [\hm(\wedge^* \I, \h)]^{\so}$. We compute the Lie-algebra cohomology groups $[H^{*}_{\rm{Lie}}(\I;\,\h)]^{so(p,\,q)}$ from the complex:
\begin{equation*}
	\begin{split}
		&[\hm(\boldsymbol{R},\, \h)]^{\so} \xrightarrow{\delta} [\hm(\I,\,\h)]^{\so}\xrightarrow{\delta} [\hm(\I^{\wedge2},\,\h)]^{\so}\xrightarrow{\delta}\\
		& \cdots \xrightarrow{\delta}[\hm(\I^{\wedge (n-2)},\h)]^{\so} \xrightarrow{\delta}[\hm(\I^{\wedge (n-1)},\h)]^{\so}\\ &\xrightarrow{\delta} [\hm(\I^{\wedge n},\,\h)]^{\so}.	
	\end{split}
\end{equation*}
$\delta I = 0$, $\delta \rho = 0$, $\delta \beta=(-1)^{(n-1)}(n-1)\gamma$ and  $\delta\gamma=0.$ It follows that,  
\begin{equation*}
	\begin{split}
		&[H^{0}_{{\rm{Lie}}}(\I;\, \h)]^{\so}\simeq [\hm(\boldsymbol{R},\, \h)]^{\so} \simeq \{0\}.\\
		&[H^{1}_{{\rm{Lie}}}(\I;\, \h)]^{\so}\simeq [\hm(\I,\,\h)]^{\so} \simeq \langle I \rangle.\\ &[H^{2}_{{\rm{Lie}}}(\I;\, \h)]\simeq [\hm(\I^{\wedge2},\,\h)]^{\so} \simeq \langle \rho \rangle.\\ 
		&[H^{n-2}_{{\rm{Lie}}}(\I;\, \h)]^{\so} \simeq \{0\}.\\ &[H^{n-1}_{{\rm{Lie}}}(\I;\, \h)]^{\so} \simeq \langle \gamma \rangle / \langle \gamma \rangle \simeq \{0\}.
	\end{split}
\end{equation*}
We conclude by Lemma \ref{Serre2},
$$ H^{*}_{\rm{Lie}}(\h;\,\h)\simeq H^{*}_{\rm{Lie}}(\so;\,\br) \otimes \langle I, \rho \rangle, \ \rm{where} \ I \ and \ \rho \rm{ \ is \  as \  defined \  above.}$$
\end{proof}
We are now ready to compute  $HL^*(\h; \h)$, we begin in low dimensions.
\begin{lemma} \label{partI}
$HL^0(\h ;\,\h)\simeq 0$ and $HL^1(\h ;\,\h)=  \langle I \rangle,$ where $I$: 
\begin{equation}
	\begin{split}
		I(\alpha_{ij}) &= 0,\ \ 1\le i<j \le p, \ \ p+1\le i<j \le n, \\
		I(\beta_{ij})& =0,\ \ 1\le i\le p,\  \ p+1\le j \le n,\\
		I(\pari)&=\pari, \ \ i=1,2,\cdots,n.
	\end{split}
\end{equation}
\begin{proof}
We compute $HL^{0}(\h;\, \h)$ and $HL^1(\h;\,\h)$  from long exact sequence \eqref{les2}. From \eqref{les2}, we have
\begin{equation*}
	0 \longrightarrow H^0_{\rm{Lie}}(\h ;\,\h)  \xrightarrow{\mathit{\pi^*_{\rm{rel}}}} HL^0(\h ;\,\h) \xrightarrow{\mathit{C_{\rm{rel}}}} 0,
\end{equation*}
which implies $\pi^*_{\rm{rel}}$ is an isomorphism and $$HL^0(\h ;\,\h)\simeq H^0_{\rm{Lie}}(\h ;\,\h) \simeq 0.$$
From \eqref{les2}, we also have
\begin{equation*}
  0 \longrightarrow H^1_{\rm{Lie}}(\h ;\,\h)  \xrightarrow{\mathit{\pi^*_{\rm{rel}}}} HL^1(\h ;\,\h) \xrightarrow{\mathit{C_{\rm{rel}}}} 0, 
\end{equation*}
$\pi^*_{\rm{rel}}$ is an isomorphism and $$HL^1(\h ;\,\h) \simeq H^1_{\rm{\rm{Lie}}}(\h ;\,\h) = \langle I \rangle,$$ where $I$ is just as defined above.
\end{proof}    
\end{lemma}
For higher dimensions, the calculations for $HL^*(\h;\,\h)$ proceed in a recursive manner, using results from lower dimensions to compute higher dimensions. Our strategy is to first find $H^*_{{\rm{rel}}}(\h)$ and then insert $H^*_{{\rm{rel}}}(\h)$ into the long exact sequence \eqref{les2} to compute $HL^{*}(\h;\, \h)$.
The Pirashvili spectral sequence converges to $H^*_{{\rm{rel}}}(\h)$  with $E_{2}$ term given by  $$E_2^{m, \, k} \simeq HL^k ( \h ; \, \h) \otimes HR^m (\h),\ \ where \ \ \ m\ge 0, \ \ \, k\ge 0 .$$ 
We demonstrate the strategy in an easy example: From the Pirashvili spectral sequence, $$E_2^{m, \, 0} \simeq HL^0 ( \h ; \, \h) \otimes HR^m (\h) \simeq 0$$ for all $m$ since $HL^0 ( \h ; \, \h) = 0.$ Since $E^{0,\,0}_{2} \simeq 0$, we have 
$H_{\rm{rel}}^{0}(\h)\simeq E^{0,\,0}_{2} \simeq 0.$ Now, we insert $H_{\rm{rel}}^{0}(\h)$ into \eqref{les2} and it yields the following:
\begin{equation*}
	0 \longrightarrow H^2_{\rm{Lie}}(\h ;\,\h)  \xrightarrow{\mathit{\pi^*_{\rm{rel}}}} HL^2(\h ;\,\h) \xrightarrow{\mathit{C_{\rm{rel}}}} H_{rel}^{0}(\h) \simeq 0.
\end{equation*}
Consequently, $HL^2(\h;\,\h) \simeq H^2_{\rm{Lie}}(\h ;\,\h) = \langle \rho \rangle,$ where $\rho$ is as defined above with $\wedge$ replaced with $\otimes$.

\begin{lemma} \label{partII}
For $n=p+q \ge 4$,  $ HL^{n}(\h; \, \h) \simeq [I \otimes\gamma^*_{pq}] $ and $ HL^{n+1}(\h) \simeq [\rho \otimes \gamma^*_{pq}] $.
\end{lemma}
\begin{proof}
We begin the iteration of elements in the $E_{2}^{*, *} $ term of the Pirashvili spectral sequence with $HL^1(\h; \, \h)$. Consider the following elements : 

$$ I \otimes \theta'   \,  \in  \, HL^1 ( \h ; \, \h) \otimes HR^m (\h) \subseteq E_2^{m, \, 1}$$ 
$$ I \otimes \gamma^*_{pq}    \,  \in HL^1 ( \h ; \, \h) \otimes HR^{n-3} (\h) \subseteq E_2^{m, \, n-3}$$ 
where $c_R(\theta')=\theta \in  H^{m+3}_{\rm{Lie}}(\so, \br)$.

In the Pirashvili spectral sequence, $d_r^{m,\,1}(I \otimes \theta')=0$ for all $r\geq 2$, so $I\otimes\theta'$ will not be a coboundary and therefore  $[I \otimes \theta'] \in H^{m+1}_{\rm{rel}}(\h)$. Using the long exact sequence \eqref{les2}, we show that $[I \otimes \theta]$ is mapped to 0 in $HL^{m+4}$. Now, 
\begin{equation*}
		\begin{split}
			\delta(I \otimes \theta')(g_{1} \otimes \cdots \otimes g_{m+4})&= \delta I \otimes \theta' \ (g_{1} \otimes \cdots \otimes g_{m+4})- I \otimes \delta \theta' \ (g_{1} \otimes \cdots \otimes g_{m+4})\\
			& + \sum\limits_{i=3}^{m+4} (-1)^i g_{i} I(g_1) \theta'(g_{2} \otimes \cdots \widehat{g_i} \cdots \otimes g_{m+4}),
		\end{split}
\end{equation*}
 $g_{i}I=0$ for all $g_{i} \in \h$ since $I$ is $\h$-invariant and  $\delta I = 0$, then $$\delta(I \otimes \theta')= - I \otimes \delta \theta' = -I \otimes c_{R} (\theta') = -I \otimes \theta.$$ Consequently, $c_{\rm{rel}}([I \otimes \theta']) = [\delta (I \otimes \theta')]= [I \otimes \theta]= [I \wedge \theta]$ in  $H_{Lie}^{m+4}(\h;\,\h)$ and $\pi^*_{{\rm{rel}}}([I \otimes \theta])=0$ in $HL^{m+4}(\h;\,\h)$.\\
 
 For  $I \otimes \gamma^*_{pq}$, it is quite easy to see that $d_r^{m,\,1}(I \otimes \gamma^*_{pq})=0$, for all $r\geq 2$ and $I \otimes \gamma^*_{pq}$ is not a coboundary in the Pirashvili spectral sequence,  therefore $[I\otimes \gamma^*_{pq}] \in H^{n-2}_{\rm{rel}}(\h)$. Now, we shift our attention to \eqref{les2} where
\begin{equation*}
		\begin{split}
			\delta(I\otimes \gamma^*_{pq})(g_1 \otimes \cdots \otimes g_{n+1}) &= \delta I \otimes \gamma^*_{pq}(g_1 \otimes \cdots \otimes g_{n+1}) - I \otimes \delta\gamma^*_{pq} (g_1 \otimes \cdots \otimes g_{n+1})\\
			& + \sum\limits_{i=3}^{n+1} (-1)^{i} g_{i}I(g_1)\gamma^*_{pq}(g_2 \otimes \cdots \widehat{g_{i}} \cdots  \otimes g_{n+1}).
		\end{split}
\end{equation*}
We know that $\delta I = 0 $, $\delta \gamma^*_{pq}=0$ and  $g_{i}I=0$ for all $g_{i} \in \h $, then $\delta(I\otimes \gamma^*_{pq}) = 0.$ 
Consequently, $c_{\rm{rel}}([I \otimes\gamma^*_{pq}]) =[\delta(I \otimes \gamma^*_{pq})]  = 0$ in  $H^{n+1}_{\rm{Lie}}(\h ;\, \h)$ which implies $$[I \otimes\gamma^*_{pq}] \in HL^{n}(\h; \, \h).$$

We continue the iteration of elements in the $E_{2}^{*, *} $ term of the Pirashvili spectral sequence with $HL^2(\h; \, \h)$. Consider these elements: 

$$ \rho \otimes \theta' \, \in HL^2 ( \h ; \, \h) \otimes HR^m (\h) \subseteq E_2^{m, \, 2}$$
$$ \rho \otimes \gamma^*_{pq}  \, \in HL^2 ( \h ; \, \h) \otimes HR^{n-3} (\h) \subseteq E_2^{m, \, n-3},$$
where $c_R(\theta')=\theta \in  H^{m+3}_{\rm{Lie}}(\so, \br)$.
 
In the Pirashvili spectral sequence, $d_{r}^{m,\,2}(\rho \otimes \theta')=0$ for all $r\ge2$, so $\rho \otimes \theta'$ is not a coboundary and therefore $[\rho \otimes \theta' ]\in H^{m+2}_{{\rm{rel}}}(\h).$ Using the long exact sequence \eqref{les2}, we show that $[\rho \otimes \theta]$ is mapped to 0 in $HL^{m+5}$. Now, 
\begin{equation*}
	\begin{split}
		\delta(\rho \otimes \theta')(g_1 \otimes \cdots \otimes g_{m+5}) &= \delta \rho \otimes \theta '(g_1 \otimes \cdots \otimes g_{m+5}) + \rho \otimes \delta \theta '(g_1 \otimes \cdots \otimes g_{m+5})\\
		& + \sum\limits_{i=4}^{m+5} (-1)^{i}g_{i}\rho(g_{1}\otimes g_{2})\theta'(g_3 \otimes \cdots  \widehat{g_{i}} \cdots \otimes g_{m+5}),
	\end{split}
\end{equation*}
$\delta \rho = 0 $ and the element $\theta'\in HR^{m}(\h)$ can be chosen so that $\theta '(y_{1}\otimes y_{2} \otimes...\otimes y_{m+2}) = 0$ if any $y_{i} \in \I$. Since $\rho \otimes \theta' \in E_{2}^{m,\,2}$, $d^2_{m,\,2}(\rho \otimes \theta')$  is skew symmetric in the variables $g_{3},g_{4},\cdots,g_{m+5}$, and suppose that $g_{i} \in \I$ for one and only one i $\in$ $\{4,5,\cdots,m+5\}$, then $$g_{i}\rho(g_{1}\otimes g_{2})\theta'(g_3 \otimes \cdots  \widehat{g_{i}} ... \otimes g_{m+5})= \pm g_{3}\rho(g_{1}\otimes g_{2})\theta'(g_i \otimes ...  \widehat{g_{3}} \cdots \otimes g_{m+5})=0, $$since $g_3 \in \so$-invariant and $\rho$ is an $\so$-invariant.  It follow that $$\delta(\rho \otimes \theta')=  \rho \otimes \delta \theta'= \rho \otimes c_{R}(\theta')=\rho \otimes \theta.$$
Consequently, $c_{\rm{rel}}([\rho \otimes \theta'])=[\delta(\rho \otimes \theta')]=[\rho \otimes \theta]=[\rho \wedge \theta]$ in $H_{Lie}^{m+5}(\h;\,\h)$ and $\pi^{*}_{{\rm{rel}}}([\rho \otimes \theta])=0$ in $HL^{m+5}(\h;\,\h).$

For  $\gamma^*_{pq} \in HR^{n-3}(\h)$, we have $\rho \otimes \gamma^*_{pq} \in E_{2}^{n-3,\,2}$ and $E_{2}^{n-5,\,3}\xrightarrow{d_{2}^{n-5,\,3}}E_{2}^{n-3,\,2}\xrightarrow{d_{2}^{n-3,\,2}} E_{2}^{n-1,\,1}.$
$d_{2}^{n-3,\,2}(\rho \otimes \gamma^*_{pq})=0$, $d_{2}^{n-5,\,3}$ is also a zero map, so $E_{3}^{n-3,\,2} \simeq E_{2}^{n-3,\,2}.$ 
Now, on the third page of the Pirashvili spectral sequence we have $$ E_{3}^{n-6,\,4}\xrightarrow{d_{3}^{n-3,\,2}} E_{3}^{n-3,\,2}\xrightarrow{d_{3}^{n-3,\,2}} E_{3}^{n,\,0}\simeq \{0\}. $$
Thus, $d_r^{n-3,\,2}(\rho \otimes \gamma^*_{pq})=0 \ {\rm{for \ all}} \ r \geq 3$ and $\rho \otimes \gamma^*_{pq}$ is not a coboundary the Pirashvili spectral sequence, which implies $[\rho \otimes \gamma^*_{pq}] \in H^{n-1}_{\rm{rel}}(\h)$. Again, we shift our attention to \eqref{les2}, $$c_{\rm{rel}}([\rho \otimes \gamma^*_{pq}]) = [\delta(\rho \otimes \gamma^*_{pq})] = 0$$ in $H^{n+2}_{\rm{Lie}}(\h;\,\h)$ which then implies $$[\rho \otimes \gamma^*_{pq}]  \in HL^{n+1}(\h).$$
We have established that $[I \otimes\gamma^*_{pq}] \in HL^{n}(\h; \, \h)$  and $[\rho \otimes \gamma^*_{pq}]  \in HL^{n+1}(\h; \, \h)$.\\
\end{proof}
\begin{theorem}
	For  $p+q\ge4$, $HL^{*}(\h;\,\h) \simeq \langle I, \rho \rangle \otimes T(\gamma^{*}_{pq})$, where\\ $T(\gamma^{*}_{pq}):=$$\sum\limits_{k\ge0}\langle\gamma^{*}_{pq}\rangle^{\otimes k}$ is the tensor algebra on the class of 
\begin{equation}
		\begin{split}
			\gamma^{*}_{pq}=&\sum\limits_{ 1\le i<j \le p} (-1)^{i+j} \dparone \wedge  \cdots \wedge \widehat{\dpari} \wedge \cdots \wedge \widehat{\dparj} \wedge \cdots  \wedge \dparn \otimes \alpha^{*}_{ij}\\
			-&\sum\limits_{ p+1\le i<j \le n} (-1)^{i+j+1} \dparone \wedge  \cdots \wedge\widehat{\dpari}
			 \wedge \cdots \wedge\widehat{\dparj} \wedge \cdots \wedge \dparn \otimes \alpha^{*}_{ij}\\
			-&\underset{p+1\le j \le n}{\sum\limits_{1\le i \le p}} (-1)^{i+j+1} \dparone \wedge  \cdots \wedge \widehat{\dpari} \wedge \cdots \wedge \widehat{\dparj} \wedge \cdots \wedge \dparn \otimes \beta^{*}_{ij}
		\end{split}
\end{equation}
\end{theorem}

\begin{proof} The result for $k=0$  and $k=1$  follow from Lemma \ref{partI} and Lemma \ref{partII} respectiveely.  We continue with iteration of elements in the $E_2^{*,*}$  term of the Pirashvili spectral sequence.\\

Consider $[I \otimes\gamma^*_{pq}] \in HL^{n}(\h; \, \h)$ and $\theta' \in HR^{m}(\h)$,
$$[I \otimes\gamma^*_{pq}] \otimes \theta' \in HL^{n}(\h; \, \h)\otimes HR^m(\h) \subseteq E^{m,\,n}_2 .$$
In the Pirashvili spectral sequence, $d^{m,n}_{r}((I \otimes\gamma^*_{pq})\otimes\theta')= 0 $ for $2\le r \le n-1$, but on page $n$ we have $$d^{m,\, n}_n( (I \otimes \gamma^*_{pq}) \otimes  \theta') = I \otimes (\theta  \otimes  \gamma^*_{pq}) \in HL^{1}(\h; \, \h)  \otimes HR^{m+n}(\h) \subseteq E^{m+n,\,1}_{n}.$$
We conclude $[I \otimes (\theta \ \otimes  \gamma^*_{pq})]$ $\notin$ $H^*_{\rm{rel}}(\h)$, when $\theta \in H^{m+3}_{Lie}(\so;\, \br ) \, {\rm{and}} \, m=0,1,2, \cdots$.\\
Note: The case of $\theta \otimes  \gamma^*_{pq} \in HR^{m+n}(\h)$, {\rm{where}} $\theta \in H^{m+3}_{Lie}(\so;\, \br ) \, {\rm{and}} \, m=0,1,2,\cdots $, with  $HL^{1}(\h; \, \h)$ \big( $I \otimes (\theta \otimes \gamma^*_{pq}) \in HL^{1}(\h; \, \h)\otimes HR^{m+n}(\h) \subseteq E^{m+n,\,1}_{2} $ \big) is now covered in the iteration.\\

With  $[I \otimes\gamma^*_{pq}] \in HL^{n}(\h; \, \h)$ and $\gamma^*_{pq} \in HR^{n-3}(\h)$, we have $[I \otimes\gamma^*_{pq}]\otimes\gamma^*_{pq} \in HL^{n}(\h; \, \h)  \otimes HR^{n-3}(\h) \subseteq E^{n-3,\,n}_{2}.$ In the Pirashvili spectral sequence, $d^{n-3,\, n}_{r}( [I \otimes\gamma^*_{pq}]\otimes\gamma^*_{pq}) =0,$ for all $r\ge 2$, which implies  $[I \otimes\gamma^*_{pq}]\otimes\gamma^*_{pq} \in$ $H^{2n-3}_{{\rm{rel}}}(\h).$
Now,
\begin{equation*}
\begin{split}
		\delta&((I \otimes\gamma^*_{pq})\otimes\gamma^*_{pq})(g_1\otimes \cdots \otimes g_n \otimes \cdots \otimes g_{2n})\\
		&=\delta(I \otimes\gamma^*_{pq})\otimes \gamma^*_{pq} (g_1\otimes \cdots \otimes g_n \otimes \cdots \otimes g_{2n})\\
		&+(I \otimes\gamma^*_{pq})\otimes \delta\gamma^*_{pq}(g_1\otimes \cdots \otimes g_n \otimes \cdots \otimes g_{2n})\\
		&+\sum\limits_{i=n+2}^{2n}(-1)^{i}(g_{i}(I \otimes\gamma^*_{pq}))(g_{1}\otimes \cdots \otimes g_{n})\gamma^*_{pq}(g_{n+1} \otimes \cdots \otimes \hat{g_{i}}\otimes \cdots \otimes g_{2n}),
	\end{split}
\end{equation*}
since $\delta(I \otimes\gamma^*_{pq}) = 0 $, $\delta \gamma^*_{pq} = 0$ , $I$ and $\gamma^*_{pq}$ are $\h$-invariant, then $\delta((I \otimes\gamma^*_{pq})\otimes\gamma^*_{pq})= 0$. Consequently, $c_{\rm{rel}}(\big[[I \otimes\gamma^*_{pq}]\otimes\gamma^*_{pq}\big]) = [\delta(I \otimes\gamma^*_{pq})\otimes\gamma^*_{pq})]=0$ in $H^{2n}_{{\rm{Lie}}}(\h;\,\h)$, so $[I \otimes\gamma^*_{pq}]\otimes\gamma^*_{pq}$ $\in$ $HL^{2n-1}(\h;\,\h).$

Using the same argument above for $[\rho \otimes \gamma^*_{pq}] \otimes \theta ' \in HL^{n+1}(\h; \, \h)  \otimes HR^{m}(\h) \subseteq E^{m,\, n+1}_{2}$ and  $\rho \otimes (\theta \ \otimes  \gamma^*_{pq}) \in HL^{2}(\h; \, \h)  \otimes HR^{m+n}(\h) \subseteq E^{m+n,\,2}_{2}$ : On page $n$ of the Pirashvilli spectral sequence,
$$d^{m,\, n+1}_{n}( (\rho \otimes \gamma^*_{pq}) \otimes  \theta') = \rho \otimes (\theta  \otimes  \gamma^*_{pq}). $$ Therefore, $[\rho \otimes (\theta \ \otimes  \gamma^*_{pq})]$ is not in $H^*_{\rm{rel}}(\h)$,  when $\theta \in H^{m+3}_{\rm{Lie}}(\so;\,\br)$ and $m=0,1,2, \cdots $.

Similarly, the case of \ $\theta \otimes  \gamma^*_{pq} \in HR^{m+n}(\h)$, {\rm{where}} $\theta \in H^{m+3}_{Lie}(\so;\, \br ) \, {\rm{and}} \, m=0,1,2,\cdots $, with  $HL^{2}(\h; \, \h)$ \big( $\rho \otimes (\theta \otimes \gamma^*_{pq}) \in HL^{2}(\h; \, \h)\otimes HR^{m+n}(\h) \subseteq E^{m+n,\,2}_{2} $ \big) is now covered in the iteration.\\

Consider $[\rho \otimes \gamma^{*}_{pq}] \in HL^{n+1}(\h;\,\h)$ and $\gamma^{*}_{pq}$ $\in$ $HR^{n-3}(\h)$, we have $(\rho \otimes \gamma^{*}_{pq}) \otimes \gamma^{*}_{pq}$ $\in$ $E_{2}^{n-3,\,n+1} $. Now, 
\begin{equation*}
	\begin{split}
			\delta((\rho \otimes \gamma^{*}_{pq}) \otimes \gamma^{*}_{pq})(g_{1} \otimes \cdots \otimes g_{2n+1})&\\
			= \delta(\rho \otimes \gamma^{*}_{pq}) \otimes \gamma^{*}_{pq} (g_{1} \otimes \cdots \otimes g_{2n+1})&\\
			+ (\rho \otimes \gamma^{*}_{pq}) \otimes \delta\gamma^{*}_{pq} (g_{1} \otimes \cdots \otimes g_{2n+1})& \\ 
			+\sum\limits_{i=n+3}^{2n+1}(-1)^{i}g_{i}\rho(g_1 \otimes g_2)\gamma^{*}_{pq}(g_3\otimes \cdots\otimes g_{n+1})&\gamma^{*}_{pq}(g_{n+2} \otimes \cdots \hat{g_{i}}\cdots \otimes g_{2n+1}).
	\end{split}
\end{equation*}
$\delta((\rho \otimes \gamma^{*}_{pq}) \otimes \gamma^{*}_{pq})(g_{1} \otimes \cdots \otimes g_{2n+1}) = \sum\limits_{i=n+3}^{2n+1}(-1)^{i}g_{i}\rho(g_1 \otimes g_2)\gamma^{*}_{pq}(g_3\otimes \cdots\otimes g_{n+1})\gamma^{*}_{pq}(g_{n+2} \otimes \cdots \hat{g_{i}} \cdots \otimes g_{2n+1})$, since $\delta(\rho \otimes \gamma^{*}_{pq})= 0$ and $\delta\gamma^{*}_{pq}=0.$
Using a skew symmetric argument, $\delta((\rho \otimes \gamma^{*}_{pq}) \otimes \gamma^{*}_{pq})=0$, which implies $(\rho \otimes \gamma^{*}_{pq})\otimes \gamma^{*}_{pq} \in H^{2n-2}_{\rm{rel}}(\h) $. Consequently, $\pi^{*}_{\rm{rel}}\big(\big[[\rho \otimes \gamma^{*}_{pq}] \otimes \gamma^{*}_{pq}\big]\big)= 0 $. Therefore  $(\rho \otimes \gamma^{*}_{pq})\otimes \gamma^{*}_{pq}$ $\in$ $HL^{2n}(\h;\,\h)$.

Also, in the Pirashvili spectral sequence, we have
$$d^n((I \otimes (\gamma^{*}_{pq})^{\otimes2})\otimes \theta')= (I \otimes \gamma^{*}_{pq} ) \otimes ( \gamma^{*}_{pq} \otimes \theta)$$
$$d^n((\rho \otimes (\gamma^{*}_{pq})^{\otimes2})\otimes \theta')= (\rho \otimes \gamma^{*}_{pq}) \otimes ( \gamma^{*}_{pq} \otimes \theta) $$

By induction on k, $HL^{*}(\h; \, \h)$ is the direct sum of vector spaces $\langle I, \rho \rangle \otimes (\gamma^{*}_{pq})^{\otimes k}.$ We conclude that $$HL^{*}(\h;\,\h) \simeq \langle I, \rho \rangle \otimes T(\gamma^{*}_{pq})$$
\end{proof}


\newpage
\begin{thebibliography}{}

\bibitem{Loday1992}  Loday, J.-L., {\em Cyclic Homology}, Spring Verlag, Heidelberg, 1992.

\bibitem{Bloh} Bloh, A., ``On a Generalization of a Concept of Lie Algebras," {\em Dokl. Akad. Nauk.}, SSSR, Vol. 165 
(1965), pp. 471--473.

\bibitem{Fil-Mandal} Fialowski, A., Mandal, A. ``Leibniz Deformations of a Lie Algebra," Journal of Mathematical Physics, Vol. 49, 
9 (20008), doi.org/10.1063/1.2981562.  

\bibitem{Fil-Mandal-Muk} Fialoswki, A., Mandal, A., Mukherjee, G., ``Versal Deformations of Leibniz Algebras," {\em K-Theory}, 
Vol. 3, 2 (2009), pp.  327--358.

\bibitem{biyogmam2011leibniz} Biyogmam, G. R., ``On the Leibniz (Co)homology of the Lie Algebra of the Euclidean Group," {\em
Journal of Pure and Applied Algebra}, Vol. 215 (2011), pp. 1889--1901.

\bibitem{lodder1998leibniz} Lodder, J. ``Leibniz Cohomology for Differentiable Manifolds," {\em Ann. Institut Fourier} (Grenoble),
Vol. 48, 1 (1998), pp.  73--95. 

\bibitem{loday1993universal} Loday, J.-L., Pirashvili, T., ``Universal Enveloping Algebras of Leibniz Algebras and (Co)homology,"
{\em Math. Ann.}, Vol. 296 (1993), pp.  139--158.

\bibitem{pirashvili1994leibniz} Pirahsvili, T., ``On Leibniz Homology," {\em Ann. Institut Fourier} (Grenoble), Vol. 44, 2 (1994), pp. 401--411.

\bibitem{Feldvoss-Wagemann} Feldvoss, J., Wagemann, F., ``On Leibniz Cohomology," {\em Journal of Algebra}, Vol 569, pp. 276 - 317, 2021. 

\bibitem{hochschild1953cohomology} Hochschild, G., Serre, J-P., ``Cohomology of Lie Algebras," {\em Annals of Mathematics},
Vol 57, 3 (1953), pp. 591--603. 

\bibitem{hall2015lie} Hall Brian., {\em Lie groups, Lie algebras, and representations: an elementary introduction}, volume 222, Springer, 2015. 

\bibitem{weibelintroduction} Weibel, C.,  {\em An introduction to homological algebra (1994) Cambridge Studies in Advanced Mathematics, 38}, Cambridge Univ. Press, Cambridge.  

\bibitem{biyogmam2013leibniz} Biyogmam, Guy., ``Leibniz Homology of the Affine Indefinite Orthogonal Lie Algebra,"  {\em African Diaspora Journal of Mathematics. New Series}, Vol 19(1), pp. 37 - 48, 2016.

\bibitem{lodder2020leibniz} Lodder, Jerry., ``Leibniz Cohomology and Connections on Differentiable Manifolds," {\em Differential Geometry and its Applications}, Vol. 79, 2021. 

\end{thebibliography}
\end{document}